\documentclass[12pt]{amsart}
\usepackage[totalwidth=480pt, totalheight=640pt]{geometry}
\usepackage{amsmath, amssymb, amsthm, amsfonts, %mathrsfs,
eucal, enumerate, layout}
\usepackage[numbers]{natbib}
\usepackage[normalem]{ulem}
\usepackage[usenames]{color}
\usepackage{verbatim}
\usepackage{soul}
\usepackage{mathtools}
\usepackage{graphicx}
\usepackage{hyperref}
\hypersetup{colorlinks=true,citecolor={red},linkcolor={blue},urlcolor={violet}}
\usepackage[new]{old-arrows}

\newcommand{\details}[1]{}

\newtheorem{theorem}{Theorem}[section]
\newtheorem*{theorem*}{Theorem}

\newtheorem*{corollary*}{Corollary}
\newtheorem{lemma}[theorem]{Lemma}
\newtheorem*{lemma*}{Lemma}

\newtheorem*{claim*}{Claim}

\newtheorem{proposition}[theorem]{Proposition}
\newtheorem*{proposition*}{Proposition}

\newtheorem*{conjecture*}{Conjecture}
\newtheorem{def-proposition}[theorem]{Definition-Proposition}
\theoremstyle{definition}

\newtheorem*{definition*}{Definition}
\newtheorem{remark}[theorem]{Remark}

\newtheorem*{example*}{Example}
\numberwithin{equation}{section}

\newcommand{\rme}{\mathrm {e}}
\newcommand{\rmi}{\mathrm {i}}

\newcommand{\ZZ}{\mathbb{Z}}
\newcommand{\QQ}{\mathbb{Q}}

\newcommand{\CC}{\mathbb{C}}
\newcommand{\GG}{\mathbb{G}}
\newcommand{\PP}{\mathbb{P}}

\newcommand{\Galmot}{{\mathcal{G}}{\mathrm{al}}_{\mathrm{mot}}}

\newcommand{\End}{\mathrm{End}}
\newcommand{\Hom}{\mathrm{Hom}}

\newcommand{\UR}{\mathrm{UR}}
\newcommand{\W}{\mathrm{W}}

\newcommand{\Lie}{\mathrm{Lie}\,}

\newcommand{\oQQ}{\overline{\QQ}}

\newcommand{\cE}{\mathcal{E}}

\begin{document}

\title[Algebraic independence of the exponential and Weierstrass $\wp$-functions]
{Algebraic independence of \\ the exponential and Weierstrass $\wp$-functions}

\author{Cristiana Bertolin}
\address{Dipartimento di Matematica, Universit\`a di Padova, Via Trieste 63, Padova}
\email{cristiana.bertolin@unipd.it}



\date{\today}

%\commby{}

%%% ----------------------------------------------------------------------

\begin{abstract}
Let $\Omega$ be a lattice in $\CC$ with algebraic invariants and complex
multiplication, let $\cE$ be the elliptic curve associated with $\Omega$,
and let $\wp$ be the Weierstrass function relative to $\Omega$. Set
$k:=\operatorname{End}(\cE)\otimes_{\ZZ}\QQ.$
We prove that if $t_1,\dots,t_s$ are $\QQ$-linearly independent
algebraic numbers and $p_1,\dots,p_n$ are $k$-linearly independent
algebraic numbers, then the $s+n$ numbers
\[
\rme^{t_1},\dots,\rme^{t_s},
\wp(p_1),\dots,\wp(p_n)
\]
are algebraically independent over $\oQQ$.
The proof uses the Tannakian description of the Lie algebra of the
unipotent radical of the $1$-motive associated with these points.
\end{abstract}

%%% ----------------------------------------------------------------------

\maketitle

%%% ----------------------------------------------------------------------

\tableofcontents

\section*{Introduction}

Let $\oQQ$ be the algebraic closure of the field $\QQ$ of rational
numbers. Let $\Omega=\ZZ\omega_1+\ZZ\omega_2$ be a lattice in $\CC$ with
elliptic invariants. Let $\cE$ be the elliptic curve
associated with $\Omega$ and denote by
$k:=\End(\cE)\otimes_{\ZZ}\QQ $
its endomorphism field. Let $\wp$ be the Weierstrass function relative
to the lattice $\Omega$.

 We prove the following algebraic independence theorem:

\begin{theorem}\label{teo:main}
	Let $\Omega$ be a lattice in $\CC$ with algebraic invariants and
	complex multiplication. If
	\begin{itemize}
		\item $t_1,\dots,t_s$ are $\QQ$--linearly independent algebraic
		numbers, and
		\item $p_1,\dots,p_n$ are $k$--linearly independent algebraic
		numbers,
	\end{itemize}
	then the $s+n$ numbers
	\[
	\rme^{t_1},\dots,\rme^{t_s},
	\wp(p_1),\dots,\wp(p_n)
	\]
	are algebraically independent over $\oQQ$.
\end{theorem}

Theorem~\ref{teo:main} simultaneously extends the following two
classical theorems:

\begin{itemize}
	\item \emph{The Lindemann--Weierstrass Theorem} (1885): if
	$t_1,\dots,t_s$ are $\QQ$--linearly independent algebraic numbers,
	then the numbers
	$\rme^{t_1},\dots,\rme^{t_s}$ are algebraically independent.
	
	\item \emph{The Philippon--W\"ustholz Theorem} (1983): let $\wp$
	be a Weierstrass function with algebraic invariants and complex
	multiplication. If $p_1,\dots,p_n$ are $k$--linearly independent
	algebraic numbers, then the numbers
	$\wp(p_1),\dots,\wp(p_n)$ are algebraically independent.
\end{itemize}

Theorem~\ref{teo:main} is a special case of the semi-elliptic
LW Conjecture for linearly independent complex numbers
\cite[Conjecture~0.6]{Bsubmitted}, or of the split semi-elliptic
LW Conjecture \cite[Conjecture~2.3]{BW}. These conjectures are, in
turn, consequences of the semi-elliptic Conjecture
\cite[Conjecture~0.2]{Bsubmitted}, which is equivalent to the
Grothendieck--Andr\'e period Conjecture applied to $1$-motives whose
abelian part is a power of an elliptic curve, and is expected to
encompass all reasonable statements concerning the values of the
exponential function, the Weierstrass $\wp$-- and $\zeta$--functions,
and Serre's $f_q$--functions.

In the statement of Theorem~\ref{teo:main}, we assume that the lattice
has complex multiplication in order to apply the
Philippon--W\"ustholz Theorem. We emphasize, however, that the complex
multiplication assumption is not used in the argument itself: it
enters only through the algebraic independence result provided by the
Philippon--W\"ustholz Theorem. Thus, if the corresponding algebraic
independence statement for the values
$\wp(p_1),\dots,\wp(p_n)$
were known without the complex multiplication assumption, the proof
given below would apply verbatim and yield the same conclusion in the
general case. This is also consistent with the
Grothendieck--Andr\'e period Conjecture, which predicts the
corresponding algebraic independence statement without the complex
multiplication hypothesis (see
\cite[Conjecture~0.6]{Bsubmitted} or
\cite[Conjecture~2.3]{BW}).

Consider the $1$-motive
\[
M=
[u:\ZZ\longrightarrow\GG_m^s\times\cE^n],
\qquad
u(1)=
(\rme^{t_1},\dots,\rme^{t_s},P_1,\dots,P_n),
\]
where
$ P_i=\exp_{\cE}(p_i)
=
[\wp(p_i):\wp'(p_i):1]
\in\cE(\CC)\subseteq\PP^2(\CC).$
The motivic Galois group $\Galmot(M)$ fits into the exact sequence
\[
0\longrightarrow
\UR(M)
\longrightarrow
\Galmot(M)
\longrightarrow
\Galmot(\cE)
\longrightarrow0,
\]
where $\UR(M)$ is its unipotent radical and $\Galmot(\cE)$ is the
motivic Galois group of $\cE$, i.e.\ the maximal reductive quotient
of $\Galmot(M)$.

Let $B$ be the smallest abelian subvariety, modulo isogenies, of
$\cE^n$ containing a multiple of the point
$(P_1,\dots,P_n)\in\cE^n(\CC)$, and let $Z(1)$ be the smallest
subtorus of $\GG_m^s$ containing the point
$(\rme^{t_1},\dots,\rme^{t_s})\in\GG_m^s(\CC)$.
According to Proposition~\ref{theorem de B03}, the Lie algebra of the
unipotent radical of $M$ splits as
\[
\Lie\UR(M)\simeq B\oplus Z(1).
\]

This splitting reflects the complete separation of the abelian and
toric parts of the unipotent radical and is the key structural
ingredient in the proof. Using the Tannakian symmetric algebra associated with
$\Lie\UR(M)$ and the square-zero property of its action on the
associated graded $1$-motive, we first show that any polynomial
relation between the two families
$\{\rme^{t_j}\}_{j=1,\dots,s}$ and
$\{\wp(p_i)\}_{i=1,\dots,n}$
is necessarily of degree at most one. Thus every such relation is the sum, up to a constant, of an
elliptic part and a toric part, each of degree at most one. The Tannakian decomposition induced by the splitting
$\Lie\UR(M)=B\oplus Z(1)$ then separates these two parts, which are
ruled out by the Philippon--W\"ustholz and Lindemann--Weierstrass
theorems, respectively.

%---------------------------------------------------------
%\section*{aknowledgement}

%--------------------------------------------------------------
\section*{Notation and conventions}

Throughout the paper, we use lower-case letters for elliptic
logarithms of points of $\cE(\CC),$ which are denoted
by the corresponding capital letters. Thus, if
$P=\exp_{\cE}(p)$, then, under the Weierstrass projective embedding,
$P=[\wp(p):\wp'(p):1]$.

For each $i=1,\dots,n$, let
$P_i=\exp_{\cE}(p_i)=[\wp(p_i):\wp'(p_i):1]$, with $p_i\in\CC\setminus\Omega$.
We have
$x(P_i)=\wp(p_i)$, where
$x\in\overline{\QQ}(\cE)$ denotes the Weierstrass $x$-coordinate.
Consequently, a polynomial
$F(X_1,\dots,X_n,Y_1,\dots,Y_s)$ in the values
$\wp(p_1),\dots,\wp(p_n),\rme^{t_1},\dots,\rme^{t_s}$
may be regarded as the rational function
\[
F(x_1,\dots,x_n,Y_1,\dots,Y_s)
\]
on $\cE^n\times\GG_m^s$, where $x_i$ denotes the pull-back of the
Weierstrass $x$-coordinate from the $i$-th factor of $\cE^n$.
In particular,
$
F\bigl(\wp(p_1),\dots,\wp(p_n),
\rme^{t_1},\dots,\rme^{t_s}\bigr)
=
F\bigl(x(P_1),\dots,x(P_n),
\rme^{t_1},\dots,\rme^{t_s}\bigr).
$
%The substitution
%\[
%X_i\longmapsto x_i,\qquad Y_j\longmapsto Y_j
%\]
%induces an injective homomorphism
%\[
%\overline{\QQ}[X_1,\dots,X_n,Y_1,\dots,Y_s]
%\hookrightarrow
%\overline{\QQ}(\cE^n\times\GG_m^s).
%\]
%Indeed, the Weierstrass coordinate
%$x:\cE\to\PP^1$ is dominant, and hence so is
%\[
%(x_1,\dots,x_n):
%\cE^n\longrightarrow(\PP^1)^n.
%\]
%Consequently, the functions
%$x_1,\dots,x_n,Y_1,\dots,Y_s$
%are algebraically independent over $\overline{\QQ}$.

%In particular, if
%\[
%F(x_1,\dots,x_n,Y_1,\dots,Y_s)
%\]
%is constant as a rational function on
%$\cE^n\times\GG_m^s$, then the original polynomial
%\[
%F(X_1,\dots,X_n,Y_1,\dots,Y_s)
%\]
%is constant.
%Thus, a polynomial relation among
%$\wp(p_1),\dots,\wp(p_n),\rme^{t_1},\dots,\rme^{t_s}$
%may be viewed as the vanishing at
%$(\rme^{t_1},\dots,\rme^{t_s},P_1,\dots,P_n)$
%of the corresponding rational function on
%$\cE^n\times\GG_m^s$, obtained by replacing each $\wp(p_i)$ with the
%Weierstrass coordinate function $x_i$.

We shall work throughout under the following assumptions:
\begin{itemize}
	\item $\Omega$ is a lattice in $\CC$ with algebraic invariants
	and complex multiplication;
	\item $t_1,\dots,t_s$ are $\QQ$--linearly independent algebraic
	numbers;
	\item $p_1,\dots,p_n$ are $k$--linearly independent algebraic
	numbers.
\end{itemize}

Lindemann's theorem on the transcendence of $\pi$ implies that, since
$t_1,\dots,t_s$ are algebraic,
\[
2\pi\rmi\notin\QQ t_1+\cdots+\QQ t_s.
\]
Hence the $\QQ$-linear independence of $t_1,\dots,t_s$ is equivalent
to the multiplicative independence of
$\rme^{t_1},\dots,\rme^{t_s}$.
Similarly, \cite[Corollary~2.6]{BW} implies that the non-zero poles of
a Weierstrass $\wp$-function with algebraic invariants are
transcendental. Consequently, since $p_1,\dots,p_n$ are algebraic,
\[
\Omega\cap(kp_1+\cdots+kp_n)=\{0\}.
\]
It follows that the $k$-linear independence of
$p_1,\dots,p_n$ is equivalent to the $k$-linear independence of the
corresponding points $P_1,\dots,P_n$.

We follow closely the notation and conventions of
\cite{B03} and \cite{D90}.

 %-------------------------------------------------
 
 \section{The Lie algebra of the unipotent radical of $\Galmot (M)$}\label{decomposition}

Consider the 1-motive $M =[u:\ZZ  \rightarrow  \GG_m^s  \times \cE^n],$ with 
\[u(1)= (\rme^{t_1}, \dots, \rme^{t_s},P_1, \dots, P_n ) \in \GG_m^s  \times \cE^n (\CC).\]
The motivic Galois group $\Galmot (M)$ fits into 
the following exact sequence 
\begin{equation*}
	0 \longrightarrow \UR(M) \longrightarrow \Galmot (M) \longrightarrow \Galmot (\cE) \longrightarrow 0
\end{equation*} 
where $\UR(M)$ is its unipotent radical and $\Galmot (\cE)$ is the motivic Galois group of $\cE$, i.e. its maximal reductive quotient. Let 
\[B\]
be 
the smallest abelian sub-variety (modulo isogenies) of $\cE^n$
which contains a multiple of the point $(P_1,\dots,P_n) \in \cE^n (\CC).$  With our hypothesis we have that $\dim	B= n.$  
Let
\[Z(1)\]
be the smallest sub-torus of $\GG_m^{s}$ which contains the point $( \rme^{t_1},\dots, \rme^{t_s}) \in \GG_m^s(\CC).$  With our hypothesis we have that $\dim	Z(1)= s.$

Let
$\langle M\rangle^\otimes$ be the Tannakian category 
generated by $M$. The unit object of $\langle M\rangle^\otimes$ is the $1$-motive $
\ZZ(0)=[\ZZ \rightarrow 0].$
We denote by $
M^{\vee} \cong \underline{\Hom}(M,\ZZ(0)) $
the dual of the $1$-motive $M$, and by
$
\operatorname{ev}_M:
M\otimes M^{\vee}\rightarrow \ZZ(0)$
and
$
\delta_M:
\ZZ(0)\rightarrow M^{\vee}\otimes M
$
the morphisms in $\langle M\rangle^\otimes$ characterizing it
(cf.~\cite[(2.1.2)]{D90}).
The Cartier dual of $M$ is the $1$-motive
$
M^*=M^{\vee}\otimes\ZZ(1). $

Set  $\widetilde M \cong {\mathrm Gr}_{*}^{\W}(M)$ i.e. 
$ \widetilde M=\ZZ \oplus \cE^n \oplus \GG_m^s$
and consider the split $1$-motive
$E=W_{-1}\bigl(\underline{\End}(\widetilde M)\bigr)$ of weights $ \leqslant -1$
whose only non-zero graded pieces are
\[
\begin{aligned}
	E_{-1}&= \cE^n \otimes \cE^{* s},\\
	E_{-2}&= \GG_m^{s}.
\end{aligned}
\]
As observed in \cite[\S3.1]{B03}, composition of endomorphisms
endows the motive $E$ with a natural ring structure
$
P:E\otimes E\rightarrow E,
$
whose only non-zero component is
\[
P:E_{-1}\otimes E_{-1}\longrightarrow E_{-2},
\]
induced by the motivic Weil pairing
$\cE\otimes\cE^*\rightarrow\GG_m$.
Antisymmetrizing the product $P$ yields a Lie bracket on $E$.
Thus $E$ carries the structure of a Lie algebra $(E,[\, ,\,])$.

Let $x$ and $y_1^\vee,\ldots,y_s^\vee$ denote the standard bases of
\(\ZZ\) and of the character group \(X^*(\GG_m^{s}) :=\operatorname{Hom}(\GG_m^{s},\mathbb G_m))=\ZZ^s\) of the torus $\GG_m^s$, respectively.
The motive $E$ acts
on $\widetilde M$ via the morphism $ \rho=(\alpha_1,\alpha_2,\gamma):
E\otimes\widetilde M \rightarrow\widetilde M $
in $\langle M\rangle^\otimes$ defined by
\begin{equation} \label{eq:action}
	\begin{aligned}
		\alpha_1 =(\alpha_{1i}(-,-))_i \quad \mathrm{with} \quad \alpha_{1i} :& \cE^n \otimes  {\mathrm Gr}_{0}^{\W}(\widetilde M) \longrightarrow	\cE,\\
		\alpha_2:& \cE^{* n s} \otimes {\mathrm Gr}_{-1}^{\W}(\widetilde M)	\longrightarrow	\GG_m^{sn},\\
		\gamma \otimes X^*(\GG_m^{s}) = (\gamma_{j}(-,-,y_j^\vee))_j \quad \mathrm{with} \quad \gamma_{j}:& \GG_m^{s}\otimes {\mathrm Gr}_{0}^{\W}(\widetilde M) \otimes X^*(\GG_m^{s})	\longrightarrow	 \GG_m,
	\end{aligned}
\end{equation}
where the first and the third morphisms are  projections,
whereas the second one is induced by the Weil pairing $ \cE \otimes \cE^{*} \rightarrow \GG_m.$
By \cite[Lemma 3.3 (2)]{B03} the morphism
$\rho=(\alpha_1,\alpha_2,\gamma):
E\otimes\widetilde M
\rightarrow
\widetilde M
$ endows $\widetilde M$ with the structure of a module over the Lie algebra $(E,[\cdot,\cdot]).$

\begin{lemma}\label{RingLieStructure}
	\begin{enumerate}
		\item $(B,Z(1),[\cdot,\cdot])$ is the smallest Lie
		subalgebra of $(E,[\cdot,\cdot])$ containing the data required to
		recover the Tannakian category generated by the $1$-motive $M$ from its associated graded object 
		$\widetilde M \cong \operatorname{Gr}^{\W}_{*}(M).$
		\item $B\oplus Z(1)$ is a square-zero algebra.
		\item $(B,Z(1),[\cdot,\cdot])$ is an abelian Lie algebra.
	\end{enumerate}
\end{lemma}

\begin{proof}
	(1) This follows from \cite[\S3.7]{B03}.
	
	(2) Since $B\subseteq\cE^n$, the restriction to $B\oplus Z(1)$
	of the product
	$
	P:E_{-1}\otimes E_{-1}\rightarrow E_{-2}
	$
	vanishes identically.
	
	(3) The Lie bracket is obtained by antisymmetrizing $P$ and is
	therefore trivial by (2).
\end{proof}

Recalling that the abelian variety $B$ is contained in $\cE^n$,
the restriction $ \rho=(\alpha_1,\alpha_2,\gamma):
(B,Z(1),[\,,\,]) \otimes\widetilde M \rightarrow\widetilde M $ of the action \eqref{eq:action} to the Lie algebra
$(B,Z(1),[\,,\,])$  is 
\begin{equation}	
	\begin{aligned}\label{eq:action2}
		\alpha_{1i}(P_1, \dots,P_n,x)&=P_i,\\
		\alpha_2&=0,\\
		\gamma_j( \rme^{t_1},\dots, \rme^{t_s},x,y_j^\vee)&= \rme^{t_j}.
	\end{aligned}
\end{equation}
Explicitly
$\rho(P_1, \dots,P_n,\rme^{t_1},\dots, \rme^{t_s},x)=  u(1)= (\rme^{t_1}, \dots, \rme^{t_s},P_1, \dots, P_n ).$

\begin{remark}
	The point
	$
	\rho(P_1,\dots,P_n,\rme^{t_1},\dots,\rme^{t_s},	x)
	$
	determines the $1$-motive
	$
	M=[u:\ZZ\rightarrow\cE^n\times\GG_m^s].
	$
	More generally, let $p'_1,\dots,p'_n$ be obtained from
	$p_1,\dots,p_n$ by an invertible $k$-linear change of basis, and let
	$t'_1,\dots,t'_s$ be obtained from $t_1,\dots,t_s$ by an invertible
	$\QQ$-linear change of basis. In particular, all the $p'_i$ and
	$t'_j$ are algebraic. Their classes form, respectively, a $k$-basis
	of
	$
	\langle p_1,\dots,p_n\rangle_k
	\subset
	\CC/(\Omega\otimes_{\ZZ}\QQ)
	$
	and a $\QQ$-basis of
	$
	\langle t_1,\dots,t_s\rangle_{\QQ}
	\subset
	\CC/2\pi\rmi\QQ.
	$
	We obtain another point
	$
	u'(1)=
	(\rme^{t'_1},\dots,\rme^{t'_s},P'_1,\dots,P'_n)
	\in
	\GG_m^s\times\cE^n (\CC),
	$
	and hence another $1$-motive
	$
	M'=[u':\ZZ\rightarrow\GG_m^s\times\cE^n].
	$
	Although $M'$ need not be isomorphic to $M$, the two $1$-motives
	generate the same Tannakian category
	(see \cite[Remark~3.9(1)]{B03}).
\end{remark}

\begin{lemma}\label{uu'=0}
	For all $u,u'\in  B  \oplus Z(1), $ we have that $\rho(u)\rho(u')=0.$
\end{lemma}

\begin{proof}
	The action $ \rho=(\alpha_1,\alpha_2,\gamma):
	(B,Z(1),[\,,\,]) \otimes\widetilde M \rightarrow\widetilde M $
	is completely described by the two projection morphisms
	$\alpha_1$ and $\gamma$ of the Tannakian category $\langle M\rangle^\otimes,$
	while no intermediate morphism from the weight \(-1\) piece to the weight
	\(-2\) piece occurs, since $\alpha_2$ is identically zero. Writing $
	\widetilde M
	=
	M_0\oplus M_{-1}\oplus M_{-2},$
	with $
	M_0=\operatorname{Gr}^W_0(\widetilde M),
	M_{-1}=\operatorname{Gr}^W_{-1}(\widetilde M),
	M_{-2}=\operatorname{Gr}^W_{-2}(\widetilde M),$
	the action of an element $
	u=(b,z)\in B \oplus Z(1)$
	has block form
	\[
	\rho(u)=
	\begin{pmatrix}
		0 & 0 & 0\\
		\alpha_1(b) & 0 & 0\\
		\gamma(z) & 0 & 0
	\end{pmatrix}.
	\]
	Hence the action is concentrated on the weight $0$ component.
	In particular, for any
	$	u=(b,z),	u'=(b',z')	\in  B \oplus Z(1),$ $\rho(u')$ maps $M_0$ into $
	M_{-1}\oplus M_{-2},$
	whereas $\rho(u)$ has no non-zero component starting from either $M_{-1}$ or $M_{-2}$. Therefore
	$	\rho(u)\rho(u')=0 $
	for every $u,u'\in B\oplus Z(1)$.
\end{proof}

According to \cite[Theorem 3.8, Theorem 0.1]{B03}, 

\begin{proposition} \label{theorem de B03}
	\begin{enumerate}
		\item The associated graded functor
		$\operatorname{Gr}^{\W}_{*}:
		\langle M\rangle^{\otimes}
		\rightarrow
		\langle \widetilde M\rangle^{\otimes}$
		induces an equivalence between the Tannakian category
		$\langle M\rangle^{\otimes}$ and the category of objects of
		$\langle\widetilde M\rangle^{\otimes}$ endowed with the action of the
		Lie algebra $(B,Z(1),[\,,\,])$ defined by~\eqref{eq:action2}.
		\item The Lie algebra of the unipotent radical $\UR(M)$ of $M$ is
		the abelian Lie algebra
		\[
		\bigl(B \oplus Z(1),[\ ,\ ]=0\bigr),
		\]
		whose underlying object is a split semi-abelian variety.
	\end{enumerate}
\end{proposition}

\begin{proof} (1) See \cite[Theorem 3.8]{B03}.
	
	(2) According to \cite[Theorem 0.1]{B03} and \cite[Lemma 3.1]{BP},
	$\Lie\UR(M)$ is an extension of $B$ by $Z(1)$ whose factors of
	automorphy take values in $[B,B]$. By
	Lemma~\ref{RingLieStructure} (3), the Lie bracket on $B$ is trivial,
	so that $[B,B]=0$. Hence this extension is trivial, and therefore
	\[
	\Lie\UR(M)\cong B\oplus Z(1).
	\]
	Finally, Lemma~\ref{RingLieStructure} (3) also shows that the induced
	Lie bracket on $B\oplus Z(1)$ is trivial. 
\end{proof}

\section{The symmetric algebras and vector bundle associated with an object}

We recall the construction of \cite[\S\S 7.9--7.10]{D90}.

Let $\mathcal T$ be a Tannakian category and let
$S=\operatorname{Sp}(A)$ be an affine scheme in $\mathcal T$.
We denote by $\mathbf 1_A$ the $A$-module $A$.
Let $M$ be an $A$-module. The $n$-th symmetric power
$\operatorname{Sym}_A^n(M)$ of $M$ is the largest quotient of
$\otimes_A^n M$ on which the symmetric group $\mathfrak S_n$ acts
trivially. Deligne defines the symmetric algebra of $M$ by
\[
\operatorname{Sym}_A(M)
:=
\bigoplus_{n\geq0}\operatorname{Sym}_A^n(M),
\]
and associates to $M$ the affine $S$-scheme
\[
V(M):=
\operatorname{Sp}\bigl(\operatorname{Sym}_A(M)\bigr).
\]
For every commutative $A$-algebra $B$, the universal property of the
symmetric algebra gives a natural bijection
\[
\operatorname{Hom}_{A\text{-alg}}
\bigl(\operatorname{Sym}_A(M),B\bigr)
\simeq
\operatorname{Hom}_A(M,B).
\]
Let $T=\operatorname{Sp}(B)\to S$ and let
$
M_T:=M\otimes_A B $
be the $B$-module obtained from $M$ by base change. Then $V(M)$
represents the functor
$
T\mapsto
\operatorname{Hom}_T(M_T,\mathbf 1_T).
$
Thus $V(M)$ is the affine group $S$-scheme representing linear forms
on $M$.

Let $M^\vee$ be the dual of $M$. For every $T\to S$, there is a
natural identification
\[
\Gamma(M_T)
:=
\operatorname{Hom}_T(\mathbf 1_T,M_T)
\simeq
\operatorname{Hom}_T(M_T^\vee,\mathbf 1_T).
\]
Applying the preceding construction to $M^\vee$, the functor of
global sections
$
T\mapsto
\Gamma(M_T)
=
\operatorname{Hom}_T(\mathbf 1_T,M_T)
$
is represented by
\[
V(M^\vee)
:=
\operatorname{Sp}
\bigl(\operatorname{Sym}_A(M^\vee)\bigr).
\]
Deligne calls $V(M^\vee)$ the vector bundle associated with $M$.
Thus $V(M^\vee)$ is the affine group $S$-scheme representing the
global sections of $M$.

In particular, global sections of $M$ are represented by points of
$V(M^\vee)$, while $\operatorname{Sym}_A(M^\vee)$ is the algebra of
polynomial functions on $V(M^\vee)$.

\section{The symmetric algebras of the Lie algebra $\Lie \UR(M)$}\label{lem:symmetric-differential-operators}

From now on, we denote by $\mathbf 1$ the 1-motive $\ZZ(0)=[\ZZ \to 0]$ and by $L$ the Lie algebra $\Lie \UR(M)=(B,Z(1), [\, , \,]). $
The two algebras $\operatorname{Sym}_{\mathbf 1}(L) $ and $\operatorname{Sym}_{\mathbf 1}(L^\vee)$ have conceptually different roles:
\begin{itemize}
	\item $\operatorname{Sym}_{\mathbf 1}(L)$ is the algebra of operators on $L$,
	\item $\operatorname{Sym}_{\mathbf 1}(L^\vee)$ is the algebra of polynomial functions on $L$.
\end{itemize}

We begin by identifying the symmetric algebra
$\operatorname{Sym}_{\mathbf 1}(L)$ with an algebra of differential
operators on $\operatorname{Sym}_{\mathbf 1}(L^\vee)$. Because of the splitting $ L=B\oplus Z(1),$
one has $
\operatorname{Sym}_{\mathbf 1}(L)
\simeq
\operatorname{Sym}_{\mathbf 1}(B)\otimes
\operatorname{Sym}_{\mathbf 1}(Z(1)) .$
Hence the symmetric algebra $\operatorname{Sym}_{\mathbf 1}(L)$ has the canonical bigrading
\[ \operatorname{Sym}_{\mathbf 1}(L)= \bigoplus_{a,b\geq0}
\operatorname{Sym}^a_{\mathbf 1}(B)\otimes
\operatorname{Sym}^b_{\mathbf 1}(Z(1)).
\]
Dually the symmetric algebra $\operatorname{Sym}_{\mathbf 1}(L^\vee)$ has the canonical bigrading
\[
\operatorname{Sym}_{\mathbf 1}(L^\vee)=
\bigoplus_{r,s\geq0}
\operatorname{Sym}^r_{\mathbf 1}(B^\vee)
\otimes
\operatorname{Sym}^s_{\mathbf 1}(Z(1)^\vee).
\]

\begin{lemma}\label{derivation}
	There is a canonical action
	\[
	\phi: \operatorname{Sym}_{\mathbf 1}(L)\otimes \operatorname{Sym}_{\mathbf 1}(L^\vee) \longrightarrow \operatorname{Sym}_{\mathbf 1}(L^\vee)
	\]
	such that every element $u\in L$ acts on $\operatorname{Sym}_{\mathbf 1}(L^\vee)$ as a
	constant-coefficient derivation
\[
\begin{aligned}
	D_u:
	\operatorname{Sym}_{\mathbf 1}^r(L^\vee)
	&\longrightarrow
	\operatorname{Sym}_{\mathbf 1}^{r-1}(L^\vee),\\
	\ell_1\cdots\ell_r
	&\longmapsto
	D_u(\ell_1\cdots\ell_r)
	=
	\sum_{i=1}^r
	\operatorname{ev}_{L^\vee}(\ell_i,u)\,
	\ell_1\cdots\widehat{\ell_i}\cdots\ell_r, 
\end{aligned}
\]
	where $
	ev_{L^\vee}:
	L^\vee\otimes L
	\rightarrow \mathbf 1 $
	is the evaluation morphism in the Tannakian category $\langle M \rangle^\otimes $ (here $
	\ell_1\cdots\ell_r$
	denotes the image of $
	\ell_1\otimes\cdots\otimes\ell_r$
	in $\operatorname{Sym}_{\mathbf 1}^r(L^\vee)$).
	
	If we set $
	A_{r,s}
	:=
	\operatorname{Sym}^r_{\mathbf 1}(B^\vee)
	\otimes
	\operatorname{Sym}^s_{\mathbf 1}(Z(1)^\vee),$ for any $b\in B$ and $z\in Z(1)$, one has
	\[
	D_b:A_{r,s}\longrightarrow A_{r-1,s},
	\qquad
	D_z:A_{r,s}\longrightarrow A_{r,s-1},
	\]
	and therefore $D_bD_z:	A_{r,s} \rightarrow A_{r-1,s-1}.$
	In particular the component $	
	B\otimes Z(1)
	\subset
	\operatorname{Sym}^2_{\mathbf 1}(L) $
	acts on $\operatorname{Sym}_{\mathbf 1}(L^\vee)$ by mixed second-order differential operators.
\end{lemma}

\begin{proof}
	For every $u\in L$, the evaluation morphism  $
	ev_{L^\vee}:
	L^\vee\otimes L
	\rightarrow \mathbf 1 $
	defines a contraction $
	D_u: L^\vee\rightarrow \mathbf 1,
	\ell\mapsto \operatorname{ev}_{L^\vee}(\ell,u).$
	By the universal property of the symmetric algebra, the linear map
	$ L^\vee \rightarrow \mathbf 1 =\operatorname{Sym}_{\mathbf 1}^0(L^\vee)
	\hookrightarrow \operatorname{Sym}_{\mathbf 1}(L^\vee),
	\ell\mapsto \operatorname{ev}_{L^\vee}(\ell,u),$
	extends uniquely to a derivation
	\[
	D_u:\operatorname{Sym}_{\mathbf 1}(L^\vee)
	\longrightarrow
	\operatorname{Sym}_{\mathbf 1}(L^\vee).
	\]
	Since $D_u$ maps the generators
	$\operatorname{Sym}^1_{\mathbf 1}(L^\vee)=L^\vee$ to
	$\operatorname{Sym}^0_{\mathbf 1}(L^\vee)=\mathbf 1$, the derivation $D_u$
	has degree $-1$. Hence
	$
	D_u(FG)=D_u(F)G+F D_u(G)$
	for all $
	F,G\in\operatorname{Sym}_{\mathbf 1}(L^\vee).$ Applying the Leibniz rule to a decomposable element
$	\ell_1\cdots\ell_r
	\in
	\operatorname{Sym}^r_{\mathbf 1}(L^\vee)$
	gives
	\[
	D_u(\ell_1\cdots\ell_r)
	=
	\sum_{i=1}^r
\operatorname{ev}_{L^\vee}(\ell_i,u)\,
	\ell_1\cdots\widehat{\ell_i}\cdots\ell_r.
	\]
	In particular, $D_u\bigl(\operatorname{Sym}^r_{\mathbf 1}(L^\vee)\bigr)
	\subseteq
	\operatorname{Sym}^{r-1}_{\mathbf 1}(L^\vee).$
	Moreover, for $u,u'\in L$, the derivations $D_u$ and
	$D_{u'}$ commute: $ D_uD_{u'}=D_{u'}D_u.$
	Indeed, on a decomposable element
	$\ell_1\cdots\ell_r$, both compositions are equal to
	\[
	\sum_{i\neq j}
\operatorname{ev}_{L^\vee}(\ell_i,u)
	\operatorname{ev}_{L^\vee}(\ell_j,u')
	\ell_1\cdots
	\widehat{\ell_i}\cdots
	\widehat{\ell_j}\cdots
	\ell_r.
	\]
	Consequently, the morphism
	$L \rightarrow
	\operatorname{End}\bigl(
	\operatorname{Sym}_{\mathbf 1}(L^\vee)
	\bigr),
	u\mapsto D_u,$
	has commuting image. By the universal property of the symmetric
	algebra, it therefore extends uniquely to an algebra morphism
	$
	\operatorname{Sym}_{\mathbf 1}(L)
	\rightarrow
	\operatorname{End}\bigl(
	\operatorname{Sym}_{\mathbf 1}(L^\vee)
	\bigr).$
	
	Since $ L^\vee = B^\vee \oplus Z(1)^\vee,$
	the evaluation morphism decomposes accordingly as
	$	\operatorname{ev}_{L^\vee}
	=
	\operatorname{ev}_{B^\vee}
	\oplus
	\operatorname{ev}_{Z(1)^\vee},$
	with vanishing cross-pairings
	$
	B^\vee\otimes Z(1) \rightarrow \mathbf 1,
	Z(1)^\vee\otimes B \rightarrow \mathbf 1.$
	Therefore, if $b\in B$, its pairing with $Z(1)^\vee$ is zero.
	Hence $D_b$ acts only on the $B^\vee$-factor, and consequently
	$ D_b(A_{r,s})\subseteq A_{r-1,s}.$
	Similarly, if $z\in Z(1)$, its pairing with $B^\vee$ is zero.
	Thus $D_z$ acts only on the $Z(1)^\vee$-factor, and
	$	D_z(A_{r,s})\subseteq A_{r,s-1}.$
	It follows that
	\[
	D_bD_z(A_{r,s})\subseteq A_{r-1,s-1}.
	\]
	Finally, under the canonical decomposition
	$
	\operatorname{Sym}^2_{\mathbf 1}(B\oplus Z(1))
	\simeq
	\operatorname{Sym}^2_{\mathbf 1}(B)
	\oplus
	\bigl(B\otimes Z(1)\bigr)
	\oplus
	\operatorname{Sym}^2_{\mathbf 1}(Z(1)),
	$
	the element $b\otimes z$ acts precisely as
	$ D_bD_z.$
	Hence the mixed component $B\otimes Z(1)$ acts on
	$\operatorname{Sym}_{\mathbf 1}(L^\vee)$ by mixed second-order
	differential operators.
\end{proof}

\begin{remark}
	The action described in Lemma~\ref{derivation}
	is the intrinsic analogue of the usual action of constant-coefficient
	differential operators on a polynomial algebra.
		Indeed, consider first the case of a two-dimensional $\QQ$-vector space
	$V=	\mathbb Q e_X\oplus \mathbb Q e_Y,$
	and let
	$	X,Y\in  V^\vee$
	be the dual basis, so that
	$
	\langle X,e_X\rangle=1,
	\langle Y,e_Y\rangle=1,
	\langle X,e_Y\rangle
	=
	\langle Y,e_X\rangle
	=
	0.$
	Then $
	\operatorname{Sym}(V^\vee)
	\simeq
	\mathbb Q[X,Y].$
	By the definition of the contraction operators,
	$	D_{e_X}(X)=1,
	D_{e_X}(Y)=0,$
	and $
	D_{e_Y}(X)=0,
	D_{e_Y}(Y)=1. $
	Since both operators satisfy the Leibniz rule, one obtains
	\[
	D_{e_X}
	=
	\frac{\partial}{\partial X},
	\qquad
	D_{e_Y}
	=
	\frac{\partial}{\partial Y}.
	\]
	Consequently, the element
	$ e_Xe_Y
	\in
	\operatorname{Sym}^2(V)$
	acts on $\mathbb Q[X,Y]$ as the mixed second-order differential
	operator
	\[
	D_{e_X}D_{e_Y}
	=
	\frac{\partial^2}{\partial X\,\partial Y}.
	\]
\end{remark}

For every $r\geq 1$, the evaluation morphism
$\operatorname{ev}_{L^\vee}:L^\vee\otimes L\rightarrow\mathbf 1$
induces an evaluation pairing
\[
\operatorname{ev}_{\operatorname{Sym}^r_{\mathbf 1}(L^\vee)}:
\operatorname{Sym}^r_{\mathbf 1}(L^\vee)
\otimes
\operatorname{Sym}^r_{\mathbf 1}(L)
\longrightarrow
\mathbf 1.
\]
We use the convention that, on decomposable elements, this pairing is
given by
\begin{equation}\label{conventionEV}
\operatorname{ev}_{\operatorname{Sym}^r_{\mathbf 1}(L^\vee)}
\bigl(
\ell_1\cdots\ell_r,
a_1\cdots a_r
\bigr)
:=
\sum_{\sigma\in\mathfrak S_r}
\prod_{i=1}^r
\operatorname{ev}_{L^\vee}
\bigl(\ell_i,a_{\sigma(i)}\bigr),
\end{equation}
where $\mathfrak S_r$ denotes the symmetric group on $r$ letters,
$\ell_1,\dots,\ell_r\in L^\vee$, and
$a_1,\dots,a_r\in L$.

We now extend the action of $L$ on $\widetilde M$ to its symmetric
algebra and consider the quotient by the kernel of this action.

\begin{lemma}\label{action}
	The Lie algebra representation
	$\rho:L\longrightarrow\operatorname{End}(\widetilde M)$
	defined in \eqref{eq:action} extends uniquely to an algebra morphism
	\[
	\widetilde\rho:
	\operatorname{Sym}_{\mathbf 1}(L)
	\longrightarrow
	\operatorname{End}(\widetilde M).
	\]
	Moreover,
	$\operatorname{Sym}_{\mathbf 1}^{\geq2}(L)
	\subseteq\ker(\widetilde\rho)$.
	In particular,
	$B\otimes Z(1)\subseteq\ker(\widetilde\rho)$.
	
	Let
	\[
	Q_\rho:=
	\operatorname{Sym}_{\mathbf 1}(L)
	\big/
	\ker(\widetilde\rho).
	\]
	Then $\widetilde\rho$ induces an injective algebra morphism
	$Q_\rho\hookrightarrow\operatorname{End}(\widetilde M)$, and hence
	$Q_\rho\simeq\operatorname{Im}(\widetilde\rho)$.
	In particular, for every $b\in B$ and $z\in Z(1)$, the class
	$\overline{bz}$ of $bz$ in $Q_\rho$ is zero.
\end{lemma}

\begin{proof}
Since $L$ is an abelian Lie algebra object in
$\langle M\rangle^\otimes$, its universal enveloping algebra is
canonically isomorphic to its symmetric algebra,
	$U(L)\simeq\operatorname{Sym}_{\mathbf 1}(L)$.
	Hence, by the universal property of the enveloping algebra, the
	representation $\rho$ extends uniquely to an algebra morphism
	\[
	\widetilde\rho:
	\operatorname{Sym}_{\mathbf 1}(L)
	\longrightarrow
	\operatorname{End}(\widetilde M).
	\]
	For $v_1,\dots,v_d\in L$, the extension is given by
	\[
	\widetilde\rho(v_1\cdots v_d)
	=
	\rho(v_1)\cdots\rho(v_d).
	\]
	By Lemma~\ref{uu'=0}, one has
	$\rho(v)\rho(v')=0$ for every $v,v'\in L$.
	It follows that every element of symmetric degree at least two acts
	trivially on $\widetilde M$, and therefore
	\[
	\operatorname{Sym}_{\mathbf 1}^{\geqslant 2}(L)
	:=
	\bigoplus_{d\geq2}
	\operatorname{Sym}_{\mathbf 1}^{d}(L)
	\subseteq
	\ker(\widetilde\rho).
	\]
		Since $
	\operatorname{Sym}_{\mathbf 1}^{2}(L)
	\simeq
	\operatorname{Sym}_{\mathbf 1}^{2}(B)
	\oplus
	\bigl(B\otimes Z(1)\bigr)
	\oplus
	\operatorname{Sym}_{\mathbf 1}^{2}(Z(1)),
	$
	we obtain in particular
	$
	B\otimes Z(1)\subseteq\ker(\widetilde\rho).$
	
	Finally, the morphism induced by $\widetilde\rho$ identifies $Q_\rho =
	\operatorname{Sym}_{\mathbf 1}(L) /
	\ker(\widetilde\rho)$
	with $\operatorname{Im}(\widetilde\rho)$.
	Therefore, for every $b\in B$ and $z\in Z(1)$,
	$\overline{bz}=0 $ in $Q_\rho.$
\end{proof}

\section{Proof the main theorem}

The point $u(1)= (\rme^{t_1},\dots,\rme^{t_s},P_1,\dots,P_n),$ defining the
$1$-motive $M$, is a global section of the 1-motive $M_T$ for some affine scheme $\mathrm{Sp}(C) \to \mathrm{Sp}({\mathbf 1})$ in the Tannakian category $\langle M \rangle^\otimes ,$ that is $u(1) \in \Gamma (M_T)=V(M^\vee)(\mathrm{Sp}(C)).$ Hence the point $u(1)$ is equivalently
described by a morphism of commutative algebras in $\langle M \rangle^\otimes $
\[
\operatorname{Ev}_{u(1)}:
\operatorname{Sym}_{\mathbf 1}(M^\vee)
\longrightarrow C.
\]
This is the intrinsic Tannakian analogue of evaluation of polynomial
functions at a point.
We can therefore define the ideal of polynomial functions vanishing at $u(1)$ by
\[	I_{u(1)}
	=
	\ker\Big(
	\operatorname{Sym}_{\mathbf 1}(M^\vee)
	\overset{\operatorname{Ev}_{u(1)}}{\longrightarrow}
	C
	\Big).
\]

Consider a point $u=(b,z)$ of $L=B\oplus Z(1)$, where
$b=(P'_1,\dots,P'_n)$ and
$z=(\rme^{t'_1},\dots,\rme^{t'_s}).$
Here $p'_1,\dots,p'_n$ are obtained from $p_1,\dots,p_n$ by an
invertible $k$-linear change of basis, and $t'_1,\dots,t'_s$ are
obtained from $t_1,\dots,t_s$ by an invertible $\QQ$-linear change
of basis. In particular, all the $p'_i$ and $t'_j$ are algebraic.
Their classes form, respectively, a $k$-basis of
$\langle p_1,\dots,p_n\rangle_k$ in
$\CC/(\Omega\otimes_{\ZZ}\QQ)$ and a $\QQ$-basis of
$\langle t_1,\dots,t_s\rangle_{\QQ}$ in
$\CC/2\pi\rmi\QQ$. The point $u=(b,z)$ is a global section of the 1-motive $L_T,$ that is $u \in \Gamma (L_T)=V(L^\vee)(\mathrm{Sp}(C)).$ Hence the point $u$ is equivalently
described by a morphism of commutative algebras in $\langle M \rangle^\otimes $
\[
\operatorname{Ev}_{u}:
\operatorname{Sym}_{\mathbf 1}(L^\vee)
\longrightarrow C.
\]
As before we define the ideal of polynomial functions vanishing at $u$ by
\[	I_{u}
=
\ker \Big(
\operatorname{Sym}_{\mathbf 1}(L^\vee)
\overset{\operatorname{Ev}_{u}}{\longrightarrow}
C
\Big).
\]

We have two actions of the same symmetric algebra
$\operatorname{Sym}_{\mathbf 1}(L)$. On the one hand, by
Lemma~\ref{action}, the action
$\widetilde\rho:
\operatorname{Sym}_{\mathbf 1}(L)
\rightarrow
\operatorname{End}(\widetilde M)$
satisfies
$B\otimes Z(1)\subseteq\ker(\widetilde\rho)$.
On the other hand, by Lemma~\ref{derivation}, the canonical
differential action
$\phi:
\operatorname{Sym}_{\mathbf 1}(L)
\rightarrow
\operatorname{End}\bigl(\operatorname{Sym}_{\mathbf 1}(L^\vee)\bigr)$
associates with each $v\in L$ the constant-coefficient derivation
$\phi(v)=D_v$.

Set
\[
K_\rho:=\ker(\widetilde\rho)
\]
and define its annihilator under the differential action $\phi$ by
\[
\operatorname{Ann}_{\phi}(K_\rho)
:=
\left\{
F\in\operatorname{Sym}_{\mathbf1}(L^\vee)
\;\middle|\;
\phi(P)F=0
\text{ for every }P\in K_\rho
\right\}.
\]
A polynomial function is said to be $\rho$-compatible if it belongs
to $\operatorname{Ann}_{\phi}(K_\rho)$.
We define the subobject of $\rho$-compatible polynomial relations
vanishing at $u$ by
\[
I_u^\rho
:=
I_u\cap\operatorname{Ann}_{\phi}(K_\rho).
\]
Under the equivalence of Proposition~\ref{theorem de B03} (1), the
$\rho$-compatible polynomial relations vanishing at $u$ correspond
to the polynomial relations vanishing at $u(1)$.

We now exploit the $\rho$-compatibility condition together with the
square-zero property of the action on $\widetilde M$ to control the
degree of polynomial relations vanishing at $u(1)$.

\begin{lemma}\label{linear-relations}
Let $F\in I_u^\rho$. Then $F$ has degree at most one. More precisely,
\[
F\in
\mathbf1\oplus L^\vee
=
\mathbf1\oplus B^\vee\oplus Z(1)^\vee.
\]
Consequently, $F$ can be written uniquely as
$
F=F_B+F_Z+c,
$
where
$
F_B\in B^\vee,
F_Z\in Z(1)^\vee,
c\in\mathbf1.
$
\end{lemma}

\begin{proof}
	Write
	\[
	F=\sum_{d\geq0}F_d,
	\qquad
	F_d\in\operatorname{Sym}_{\mathbf1}^d(L^\vee).
	\]
	By Lemma~\ref{action},
	$
	\operatorname{Sym}_{\mathbf1}^{\geq2}(L)
	\subseteq
	\ker(\widetilde\rho).$ 
	Since $F$ is $\rho$-compatible, every element of
	$\ker(\widetilde\rho)$ acts trivially on $F$ through the differential
	action $\phi$. Hence, for every $d\geq2$ and every
	$v_1,\dots,v_d\in L$,
	\[
	\phi(v_1\cdots v_d)F
	=
	D_{v_1}\cdots D_{v_d}F
	=
	0.
	\]
	The component of degree $0$ of this expression is precisely
	$
	D_{v_1}\cdots D_{v_d}F_d.$
	By the definition of the differential action given in Lemma \ref{derivation} and by our convention
	for the evaluation pairing on symmetric powers \eqref{conventionEV}, one has
	\[
	D_{v_1}\cdots D_{v_d}F_d
	=
	\operatorname{ev}_{\operatorname{Sym}_{\mathbf1}^d(L^\vee)}
	\bigl(F_d,v_1\cdots v_d\bigr).
	\]
	Therefore
	$
	\operatorname{ev}_{\operatorname{Sym}_{\mathbf1}^d(L^\vee)}
	\bigl(F_d,v_1\cdots v_d\bigr)=0
	$
	for every $v_1,\dots,v_d\in L$.
	Since decomposable elements generate
	$\operatorname{Sym}_{\mathbf1}^d(L)$, the non-degeneracy of the
	evaluation pairing implies
	$F_d=0$ for every $d\geq2.$
	Thus
	\[
	F=F_0+F_1\in\mathbf1\oplus L^\vee.
	\]
	Finally, since
	$
	L^\vee=B^\vee\oplus Z(1)^\vee,$
	we obtain
	$
	F=F_B+F_Z+c,$
	with
	$
	F_B\in B^\vee,
	$
	$
	F_Z\in Z(1)^\vee
	$
	and
	$
	c\in\mathbf1.
	$
\end{proof}

The preceding decomposition leaves three possibilities: a polynomial relation may
involve only the elliptic variables, only the toric variables, or both
families separately. The classical algebraic independence theorems
exclude the first two cases.

\begin{lemma}\label{separated-relations}
	Let
	\[
	F=F_B+F_Z+c
	\]
	be a polynomial relation of degree at most one vanishing at $u(1),$ where $F_B$ depends
	only on the elliptic variables, $F_Z$ depends only on the toric
	variables, and $c\in\oQQ$.
	If $F$ is non-trivial, then both $F_B$ and $F_Z$ are non-constant.
\end{lemma}

\begin{proof}
	Suppose first that $F_Z$ is constant. Absorbing $F_Z+c$ into the
	constant term, $F$ becomes a polynomial relation involving only the
	elliptic variables. By the Philippon--W\"ustholz theorem, the numbers
	$	\wp(p_1),\dots,\wp(p_n)$
	are algebraically independent over $\oQQ$. Hence $F$ is trivial.
	
	Similarly, if $F_B$ is constant, then $F$ reduces to a polynomial
	relation involving only the toric variables. By the
	Lindemann--Weierstrass theorem,
	$
	\rme^{t_1},\dots,\rme^{t_s}$
	are algebraically independent over $\oQQ$, and therefore $F$ is again
	trivial.
	
	Consequently, if $F$ is non-trivial, both $F_B$ and $F_Z$ are
	non-constant.
\end{proof}

To deal with the remaining case $
F=F_B+F_Z+c,$
where both $F_B$ and $F_Z$ are non-constant, we need the following
Tannakian decomposition lemma. Consider the two $1$-motives
$
M_B=[u_B:\ZZ\rightarrow\cE^n]
$
and
$
M_Z=[u_Z:\ZZ\rightarrow\GG_m^s],
$
where
$
u_B(1)=b
$
and
$
u_Z(1)=z.
$
Applying the same Tannakian construction used for $M$ separately to
$M_B$ and $M_Z$, we obtain the corresponding ideals of 
polynomial relations
\[
I_b^\rho\subseteq
\operatorname{Sym}_{\mathbf1}(B^\vee)
\qquad\text{and}\qquad
I_z^\rho\subseteq
\operatorname{Sym}_{\mathbf1}(Z(1)^\vee).
\]

\begin{lemma}\label{decomposition-relations}
	Under the canonical identification
	$
	\operatorname{Sym}_{\mathbf1}(L^\vee)
	\simeq
	\operatorname{Sym}_{\mathbf1}(B^\vee)
	\otimes
	\operatorname{Sym}_{\mathbf1}(Z(1)^\vee),
	$
	one has
	\[
	I_u^\rho=I_b^\rho+I_z^\rho,
	\]
	where $I_b^\rho$ and $I_z^\rho$ denote the subobjects of
	$\rho$-compatible polynomial relations associated with $M_B$ and
	$M_Z$, respectively.
\end{lemma}

\begin{proof}
	By \cite[Lemma 2.2]{B19}  the $1$-motive
	$	M=[u:\ZZ\rightarrow\cE^n\times\GG_m^s],
	u=(1),$
	and the direct sum $M_B\oplus M_Z$ generate the same Tannakian
	category. More precisely,
	$
	\langle M\rangle^\otimes
	=
	\langle M_B\oplus M_Z\rangle^\otimes
	=
	\langle M_B,M_Z\rangle^\otimes.
	$
	Set $	L_B:=\Lie\UR(M_B)=B$ and 
		$	L_Z:=\Lie\UR(M_Z)=Z(1).$
		The corresponding decomposition of the Lie algebra of the unipotent
		radical is therefore
		$
		L=\Lie\UR(M)
		=
		L_B\oplus L_Z.
		$
		Dually,
		$
		L^\vee=L_B^\vee\oplus L_Z^\vee,
		$
		and hence
		\[
		\operatorname{Sym}_{\mathbf1}(L^\vee)
		\simeq
		\operatorname{Sym}_{\mathbf1}(L_B^\vee)
		\otimes
		\operatorname{Sym}_{\mathbf1}(L_Z^\vee).
		\]
		
		All the Tannakian constructions involved in the definition of the subobjects of $\rho$-compatible polynomial
		relations are functorial and compatible with
		this decomposition. Restricting the construction associated with
		$M$ to the factor $L_B^\vee$ gives precisely the construction associated
		with $M_B$, and hence the subobject  $I_b^\rho$. Similarly, its restriction
		to the factor $L_Z^\vee$ gives the construction associated with $M_Z$,
		and hence the subobject $I_z^\rho$.
		Conversely, the natural inclusions
		$
		L_B^\vee\hookrightarrow L^\vee $ and 
		$L_Z^\vee\hookrightarrow L^\vee
		$
		identify the polynomial relations associated with $M_B$ and $M_Z$
		with polynomial relations associated with $M$. 
		
		Therefore, under the
		canonical identification
		$
		\operatorname{Sym}_{\mathbf1}(L^\vee)
		\simeq
		\operatorname{Sym}_{\mathbf1}(L_B^\vee)
		\otimes
		\operatorname{Sym}_{\mathbf1}(L_Z^\vee),
		$
		we obtain
		$
		I_u^\rho=I_b^\rho+I_z^\rho.
		$
	\end{proof}

It remains to exclude the last possible case. By
Lemmas~\ref{linear-relations} and~\ref{separated-relations}, this
amounts to ruling out a non-trivial relation
\[
F=F_B+F_Z+c\in I_u^\rho
\]
with both $F_B$ and $F_Z$ non-constant. This follows from the
Tannakian decomposition of $\rho$-compatible relations established in
Lemma~\ref{decomposition-relations}.

\begin{proof}[Proof of Theorem~\ref{teo:main}]
Under the identification $\operatorname{Sym}_{\mathbf 1}(L^\vee)
\simeq
\operatorname{Sym}_{\mathbf 1}(B^\vee)
\otimes
\operatorname{Sym}_{\mathbf 1}(Z(1)^\vee)$, the polynomial $F=F_B+F_Z+c,$
with
$
F_B\in\operatorname{Sym}_{\mathbf 1}(B^\vee)
$
and
$
F_Z\in\operatorname{Sym}_{\mathbf 1}(Z(1)^\vee),
$
is identified with
\[
F_B\otimes 1
+
1\otimes F_Z
+
c(1\otimes1)
\]
in
$
\operatorname{Sym}_{\mathbf 1}(B^\vee)
\otimes
\operatorname{Sym}_{\mathbf 1}(Z(1)^\vee).
$
Moreover, the two subalgebras corresponding to the two factors
intersect only in the unit object:
\[
\bigl(
\operatorname{Sym}_{\mathbf 1}(B^\vee)\otimes\mathbf 1
\bigr)
\cap
\bigl(
\mathbf 1\otimes
\operatorname{Sym}_{\mathbf 1}(Z(1)^\vee)
\bigr)
=
\mathbf 1.
\]
Thus the only polynomial functions which belong simultaneously to
the $B$-part and to the $Z(1)$-part are the constant ones.
Since by Lemma~\ref{decomposition-relations}
$
I_u^\rho=I_b^\rho+I_z^\rho,$
there exist
$
G_B\in I_b^\rho
$
and
$
G_Z\in I_z^\rho
$
such that
\[
F=G_B+G_Z.
\]
On the other hand, $
F=F_B+F_Z+c,$
and hence
\[
F_B-G_B=-(F_Z+c-G_Z).
\]
The left-hand side belongs to
$
\operatorname{Sym}_{\mathbf1}(B^\vee)\otimes\mathbf1,
$
whereas the right-hand side belongs to
$
\mathbf1\otimes\operatorname{Sym}_{\mathbf1}(Z(1)^\vee).
$
By the above intersection property,
there exists $\lambda\in\mathbf1$ such that
\[
F_B-G_B=\lambda,
\qquad
F_Z+c-G_Z=-\lambda.
\]
Therefore
\[
F_B-\lambda = G_B\in I_b^\rho,
\qquad
F_Z+c+\lambda = G_Z\in I_z^\rho.
\]
By the Philippon--W\"ustholz Theorem, there are no non-trivial
polynomial relations involving only the elliptic variables. Hence
\[
F_B-\lambda=0,
\]
so that $F_B=\lambda$ is constant.
Similarly, by the Lindemann--Weierstrass Theorem, there are no non-trivial
polynomial relations involving only the toric variables. Therefore
\[
F_Z+c+\lambda=0,
\]
and hence $F_Z=-c-\lambda$ is constant. This contradicts Lemma~\ref{separated-relations}, according to which
both $F_B$ and $F_Z$ must be non-constant for any non-trivial
polynomial relation. Consequently, no non-trivial polynomial relation
$F\in I_u^\rho$ exists.
\end{proof}

%-------------------------------------------
\bibliographystyle{plain}

\end{document}